\documentclass[reqno,A4paper]{amsart}

\usepackage{amssymb}
\usepackage{amsthm}
\usepackage{graphicx}
\usepackage{xcolor}
\usepackage[normalem]{ulem}

\theoremstyle{plain}
\newtheorem{Theorem}{Theorem}[section]
\newtheorem{Lemma}[Theorem]{{ Lemma}}
\newtheorem{Corollary}[Theorem]{{ Corollary}}
\newtheorem{Definition}[Theorem]{{ Definition}}
\newtheorem{Example}[Theorem]{{ Example}}
\newtheorem{Remark}[Theorem]{{ Remark}}

\begin{document}
	
	\title[On Certain forms of Transitivities for Linear Operators]
	{On Certain forms of Transitivities for Linear Operators}

		\author{Nayan Adhikary and Anima Nagar}
	
	\address{Department of Mathematics, Indian Institute of Technology Delhi,
		Hauz Khas, New Delhi 110016, INDIA.}
	
	\email{nayanadhikarysh@gmail.com, anima@maths.iitd.ac.in}
	
	
	\begin{abstract} 
		In this article we give several characterizations for various transitivity properties for linear operators. We define a general form of `Hypercyclicity Criterion' using  a Furstenberg family $\mathcal{F}$ to characterize $\mathcal{F}$-transitive operators. In particular, we find an equivalent characterization for mixing operators. We study proximal and asymptotic relations for linear operators and prove that the difference between mixing operators and Kitai's Criterion can be presented through these relations. Finally, we find an equivalent characterization of strongly transitive  abd strongly product transitive operators.
	\end{abstract}
	\maketitle
	\smallskip
	\noindent{\bf\keywordsname{}:} {hypercyclic operator, Hypercyclicity Criterion, $\mathcal{F}$-transitive operator, $\mathcal{F}$-Transitivity Criterion, mixing, strongly transitive, strongly product transitive.}

	\section{Introduction:}

    A discrete dynamical system is a pair $(X,T),$ where $X$ is usually a Hausdorff space and $T$ is a continuous self map defined on $X.$  In topological dynamics, one of the oldest and important property is - `` transitivity''. A discrete dynamical system $(X,T)$ is called \emph{topologically transitive} or \emph{transitive} (in short) if for every pair of non-empty open subsets $U$ and $V$ of $X$, there exists $n\in \mathbb{N}$ such that $T^n(U)\cap V \neq \emptyset$. We follow the various versions of transitivity as studied in \cite{AN, Anima}, where  versions of transitivity are studied for compact  $X$. Baire category arguments show that for seperable Baire spaces $X$, $(X,T)$ is topologically transitive if and only if there exists a point $x_0\in X$ such that its orbit $\mathcal{O}_T(x_0) =  \mathcal{O}(x_0)=\{T^n(x_0):n\in \mathbb{Z}_+\}$ is dense in $X$ - the later property is called \emph{point transitivity}. And such a point $x_0\in X$ is called a \emph{transitive point}.  Stronger versions of transitivity is \emph{mixing} when for every pair of non-empty open  sets $U, V \subseteq X$ there exists an $N \in \mathbb{N}$ such that $T^n(U) \cap V \neq \emptyset$ for all $n \geq N$, and \emph{weakly mixing} when $(X \times X,T \times T)$ with the diagonal action is topologically transitive . We follow \cite{GH} for associated topological dynamics. 
    
%

    \bigskip
    
    In this article, we aim to study the versions of transitivity for  an infinite dimesional separable $F$-space $X$ with a bounded linear operator $T$ acting on it giving a linear dynamical system $(X,T)$. Note that here $X$ is never compact however for Banach $X$ \emph{every linear dynamical system     has a dynamical compactification } [Lemma 2.4 \cite{M}], where a system $(\widehat{X},\widehat{T})$ is called a \emph{dynamical compactification}  of the linear dynamical system $(X,T)$ if $\widehat{X}$ is a compact metric space, $\widehat{T}$ is continuous, $X$ is embedded as a dense subset of $\widehat{X}$, and $\widehat{T}_{|X} = T$. Note that  then 
$X$ is  a $\mathit{G}_\delta$ subset of $\widehat{X}$ by the completeness property. We observe that transitivity properties of $(\widehat{X},\widehat{T})$ can be easily passed on to $(X,T)$, though this is not a new observation.

\bigskip
       
    The collection $ \mathcal{H}(\mathbb{C}) $ of functions that are holomorphic on $ \mathbb{C} $ endowed with the compact-open topology is an $ F $-space.  \emph{G. D. Birkhoff} in 1929 showed that for each $ a \neq 0 \in  \mathbb{C} $ the translation operator     $ T_a $  on $ \mathcal{H}(\mathbb{C}) $ defined as 
    $$T_af(z) = f(z + a) \ \ (z \in \mathbb{C}, \  f \in \mathcal{H}(\mathbb{C}))$$    
    has a dense orbit \cite{birkhoff}, providing the first known example of  transitive operator on a separable $F$-space. In 1952, \emph{MacLane} \cite{gm} showed that the derivative operator has a dense orbit on $ \mathcal{H}(\mathbb{C}) $. Later \emph{Rolewicz} \cite{sr} in 1969 constructed examples of vectors with dense orbits for operators on Banach spaces.
    
    This triggered the study of linear dynamics, which formally began with the unpublished Toronto thesis of \emph{Carol Kitai} \cite{kitai} who studied  operators that admit vectors with dense orbits. She called such vectors \emph{`orbital'}. In 1984,  Beauzamy \cite{b1} introduced the term \emph{``hypercyclic”} for such vectors with dense orbits as a variation to \emph{cyclic vectors}. The study of hypercyclic vectors  formally began with \cite{b1,b2,b3}.   It is known that hypercyclic vectors cannot exist in finite dimensional topological vector spaces[Theorem 1.2 \cite{kitai}].  If $X$ is an infinite dimensional, separable $ F$-space and $T$  a bounded linear operator on $X$, then $T$ is called a \emph{hypercyclic operator} if  there exists a vector $x_0\in X$ such that its orbit $\mathcal{O}_T(x_0) $ is dense in $X$ and $x_0$ is called a \emph{hypercyclic vector} for $T$. Moreover,  the set of all hypercyclic vectors for $T$ is a dense $\mathit{G}_{\delta}$ subset of $X.$  We refer to \cite{linear, linear chaos, shapiro} for  details on linear dynamics.  
    
    \bigskip
    
     Let $\mathcal{L}(X)$  denote the set of all bounded linear operators defined on the linear space $X.$     
    Due to the linear structure, several properties of topological dynamics can be nicely formulated in linear dynamics. One of such significant properties  is the \emph{Hypercyclicity Criterion}, which gives a set of sufficient conditions of hypercyclicity: 
    \begin{Definition}\cite{linear,Bes} \label{HC}
		Let $X$ be an infinite dimensional,  separable $F$-space and $T \in \mathcal{L}(X)$. Then we say that $T$ satisfies the \emph{Hypercyclicity Criterion} with respect to the increasing sequence $(n_k)$ of positive integers if there exist two dense subsets $D_1, D_2$ of $X$ and a sequence of mappings $S_{n_k}:D_2 \to X$ such that\par 
		$(i)$ $T^{n_k} x \to 0$ for every $x\in D_1;$ \par 
		$(ii)$ $S_{n_k} y \to 0$ for every $y\in D_2;$ \par
		$(iii)$ $|| T^{n_k}S_{n_k}y - y|| \to 0$ as $k\to \infty$ for every $y\in D_2.$ \\
		We say that $T$ satisfies Hypercyclicity Criterion if there exists an increasing sequence $(n_k)$ of positive integers such that $T$ satisfies Hypercyclicity Criterion with respect to this sequence.
	\end{Definition}
	 For  an infinite dimensional, separable $F$-space $X$ and $T \in \mathcal{L}(X)$, $(X,T)$ is weakly mixing if and only if $T \oplus T$ is hypercyclic if and only if $T$ satisfies Hypercyclicity Criterion [Theorem 2.3 \cite{Bes2}].  On the other hand, Hypercyclicity Criterion with respect to the full sequence $(n)$ is known as \emph{Kitai's Criterion} since it was first discussed in \cite{kitai}.
    \begin{Definition} \cite{linear}
       A bounded linear operator $T$ defined on an infinite dimensional, separable $F$-space satisfies \emph{Kitai's Criterion} if there exist two dense subsets $D_1, D_2$ of $X$ and a sequence of mappings $S_{n}:D_2 \to X$ such that\par 
		$(i)$ $T^{n} x \to 0$ for every $x\in D_1;$ \par 
		$(ii)$ $S_{n} y \to 0$ for every $y\in D_2;$ \par
		$(iii)$ $|| T^{n}S_{n}y - y|| \to 0$ as $k\to \infty$ for every $y\in D_2.$  
    \end{Definition} 
        If $T$ satisfies Kitai's Criterion then $T$ is \emph{mixing}.
        
        \bigskip
        
         In [\cite{shapiro} page 35] Shapiro raised the question  whether every mixing operator satisfies Kitai’s Criterion? In [ Theorem 2.5 \cite{mixing example}], S. Grivaux proved that every infinite dimensional, Banach space admits a mixing operator which does not satisfy the Kitai's Criterion. Recently in [Theorem 3.5 \cite{composition}],  Gomes and Grosse-Erdmann have construted a mixing composition operator which does not satisfy Kitai's Criterion. So a natural question arises   whether there is any criterion that gives mixing operators. We answer this question affirmatively.

        Also in topological dynamics several forms of $\mathcal{F}$-transitivity can be defined using a Furstenberg family $\mathcal{F}$ and they play an important role to study dynamics of continuous maps. Analogously one can ask whether there is any form of Hypercyclicity Criterion which is equivalent to these $\mathcal{F}$-transitive systems in linear dynamics. Some such characterizations were obtained in \cite{bes ST}. We delve further into finding such equivalent characterizations of  transitive operators with respect to  Furstenberg families.\par 
        
        We recall some stronger versions of transitivity in topological dynamics studied in \cite{AN, Anima}. We discuss characterization for the same in case of linear dynamics.
       
       \bigskip 
        
        In this article we give the needed prerequisites in Section 2. In Section 3, we define Hypercyclicity Criterion using Furstenberg family $\mathcal{F}$ (Definition \ref{def1}) and then prove that this general form of Hypercyclicity Criterion is equivalent to any $\widetilde
        {\mathcal{F}}$-transitive operator (Theorem \ref{hyper}). We define several relations like asymptoticity and proximality for linear operators and then discuss their relations with $\mathcal{F}$-transitivity. Finally in  Section 4 we consider  strong transitivity  for linear operators and find its equivalent characterization (Theorem \ref{ST}). 
        
        \section{Preliminaries}

	 For the  dynamical system $(X,T)$, we consider  the \emph{hitting time sets} $N_{T}(x,V) = \{n \in \mathbb{N}: {T}^n (x) \in V\}$, $N_{T}(U,V)=\{n\in \mathbb{N}: {T}^n(U)\cap V \neq \emptyset\} = \{n\in \mathbb{N}: U \cap {T}^{-n}(V) \neq \emptyset\}$, and by taking $V=\{x\},$ the set $N_{T}(U,x)=\{n\in \mathbb{N}: x\in {T}^n(U)\}$, where $x \in X$ and non-empty $U,V  \subset X$ are  open in $X$. Recall that $ A \subset \mathbb{N} $ is called \emph{syndetic} if there exists some $M\in \mathbb{N}$ such that for every $n\in \mathbb{N},$ $\{n,n+1,\dots, n+M\}\cap A \neq \emptyset$, \emph{thick} if  for any $k\in \mathbb{N},$ there exists $n\in \mathbb{N}$ such that $\{n,n+1,\dots, n+k\}\subset A$, and \emph{cofinite} if there is $N \in \mathbb{N}$ such that $A \supset \mathbb{N} \setminus \{1,2,\ldots, N\}$.
	 
	 We start by recalling some concepts from topological dynamics following \cite{akin,AN, F2, GH}. Since these studies were essentially on compact spaces, we consider compactifications to unfold our definitions. Let $X$ be an infinite dimensional, seperable Banach space, $T \in \mathcal{L}(X)$ and $(\widehat{X},\widehat{T})$ be a dynamical compactification of $(X,T)$.  The word `opene' henceforth denotes a non-empty open set. 
	
	  The 	 system $(\widehat{X},\widehat{T})$  is \emph{topologically transitive} if the set $N_{\widehat{T}}(U,V)$ is nonempty for any two  opene $U,V \subset \widehat{X}$, and is \emph{point transitive} if there exists $x_0 \in \widehat{X}$ such that $\overline{\mathcal{O}(x_0)}$ is dense in $\widehat{X}$, equivalently the hitting time set $N_{\widehat{T}}(x_0,V)$ is non-empty for every opene $V \subset \widehat{X}$. Such a point $x_0$ is called a \emph{transitive point} and there are  $\mathit{G}_\delta$ dense transitive points for any transitive system. Since $\widehat{X}$ is complete without isolated points, the concepts of point transitivity and topological transitivity coincide here. The 	 system $(\widehat{X},\widehat{T})$  is \emph{weakly mixing} if  the set $N_{\widehat{T}}(U,V)$ is thick. The 	 system $(\widehat{X},\widehat{T})$  is  \emph{mixing} if  the set $N_{\widehat{T}}(U,V)$ is cofinite.

	  The 	 system $(\widehat{X},\widehat{T})$  is  \emph{minimal} if $\widehat{X}$ contains no proper closed invariant subset, i.e. for all $x \in \widehat{X}$ and open $U \ni x$ the set $N_{\widehat{T}}(x,U)$ is syndetic.   The property of minimality has significant importance in topological dynamics and minimality is stronger than transivity. Though linear dynamical systems can never be minimal since the origin is always a fixed point. Though, interestingly some properties of minimal systems hold for linear systems too such as Theorem \ref{RP}. Also, as can be seen in Theorem \ref{minimal},  an Auslander-Yorke Dichotomy  also holds in linear dynamics.\par 
	
	\bigskip
	
	Note that for $x \in X$ and opene  $U,V \subset \widehat{X}$, the hitting sets 
	
	$$N_{\widehat{T}}(x,V) =  N_T(x, V \cap X) = \{n \in \mathbb{N}: T^n (x) \in V \cap X\},$$	
	 $$N_{\widehat{T}}(U,V)= N_T(U\cap X, V \cap X) = \{n\in \mathbb{N}: T^n(U)\cap V \cap X \neq \emptyset\} = \{n\in \mathbb{N}: X \cap U \cap T^{-n}(V) \neq \emptyset\}$$ and by taking $V=\{x\},$ the set $$N_{\widehat{T}}(U,x)= N_T(U,x) = \{n\in \mathbb{N}: x\in T^n(U \cap X )\}.$$
	
	\bigskip
	
	We can thus extend all the above definitions to the (non compact) linear system $(X,T)$ and observe that $(X,T)$ is transitive, weakly mixing,  and mixing if and only if $(\widehat{X},\widehat{T})$  is also so.
	
Since these definitions depend only on the open sets, they can be defined for any topological vector space $X$. However, we would want point transitivity to be equivalent to transitivity and thus would require the topological vector space to be an infinite dimensional, seperable Baire space.	

\bigskip

We however note here an important distinction between transitivity on compact spaces in topological dynamics and hyperclicity on infinite dimensional, seperable $F$-space. Transitive maps on compact spaces are always surjective, whereas hypercyclic operators on infinite dimensional, seperable $F$-spaces need not be so. We consider a counterexample:

\begin{Example} \label{wei-seq}
	We recall the unilateral weighted backward shift operator as first considered in \cite{weight}.

	Let $B_w : \ell^p(\mathbb{N}) \to \ell^p(\mathbb{N})$ be the unilateral weighted backward shift operator, defined as $B_w(x_1,x_2,\dots) = (w_1x_2, w_2x_3, \dots )$ with bounded weights $w = (w_n)_{n \in \mathbb{N}}$ of real sequences. The conditions for the linear system $(\ell^p(\mathbb{N}), B_w)$ to be hypercyclic  are considered as Theorem 2.8 in \cite{ weight} and Theorem 1.40 in \cite{linear}.
	
		We take a particular case here. Consider the weight sequence 
	$$w=(w_n)= (2,\frac{1}{2},2,2,\frac{1}{4},2,2,2,\frac{1}{8},2,2,2,2,\frac{1}{16},\dots)$$
	Note that in $w$, there corresponds a subsequence $(n_k)$ on natural numbers such that $w_{n_k}=\frac{1}{2^k}$ for every $k\in \mathbb{N}.$

	 Here $\prod_{j=1}^{n_k-1} w_j = w_1\dots w_{n_k-1}=2^k$ for every $k$, 	 
	  and so $\limsup \{\prod_{j=1}^n w_j = w_1\dots w_n: n \in \mathbb{N}\}=\infty.$ 
	  
	  Let $S_{\frac{1}{w}}$ be the weighted  unilateral shift operator on $\ell^p(\mathbb{N})$, where $$S_{\frac{1}{w}}(x_1,x_2,x_3,\cdots)=(0,\frac{x_1}{w_1},\frac{x_2}{w_2},\cdots).$$ Clearly $B_wS_{\frac{1}{w}}(x)=x$ for every $x\in \ell^p(\mathbb{N}).$
	  Take $D \subset \ell^p(\mathbb{N})$ to  be the set of all sequences   which are eventually zero. Then for every $x\in D$ there exists $n\in \mathbb{N}$ such that $B_w^n(x)=0.$ On the other hand, for every $k\in \mathbb{N},$ we have $\displaystyle{\sum_{j=1}^{n_k-1}}w_j=2^k.$ So for every $x\in D,$ $S_{\frac{1}{w}}^{n_k-1}(x)\to 0$ as $k\to \infty.$ Hence $B_w$ satisfies all the conditions of \emph{Hypercyclicity Criterion} (see Definition \ref{HC}). Therefore $B_w$ is hypercyclic.

	  \par
	  
	Pick $x=(x_n) \in \ell^p(\mathbb{N})$  such  that $x_{n_k}=1/2^k$ for every $k\in \mathbb{N}$ and $x_n=0$ otherwise. Then  $ B_w^{-1}(x)= \{(r,\frac{x_1}{w_1},\frac{x_2}{w_2},\dots,\frac{x_n}{w_n},\dots): r\in \mathbb{C}\}.$ This implies that there exists no $y\in \ell^p(\mathbb{N})$ such that $y  \in T^{-1}(x)$ and so $T^{-1}(x)=\emptyset$. Therefore $T$ is not surjective. 
\end{Example}

In what follows we take our linear space $X$ to be an infinite dimensional, seperable $F$-space and $T \in \mathcal{L}(X)$.  We continue to use the term transitivity for hypercyclicity  in cognizance with topological dynamics.
	 
We recall here

 \begin{Lemma} [Theorem 4.6 \cite{linear}] \label{F1}
	A linear dynamical system $(X,T)$ is weakly mixing if and only if for every pair of open sets $U,V \subset X$ the hitting time set $N_T(U,V)$ is non-empty and for opene sets $U_1,U_2,V_1,V_2 \subset X$ there exist another pair of  opene sets $U', V'$ such that $N(U',V')\subset N(U_1,V_1)\cap N(U_2,V_2).$ 
\end{Lemma}

A pair $(x,y) \in X \times X$ is called \emph{proximal} if $\displaystyle{\liminf_{ n\to \infty}} ||T^n(x)-T^n(y)||=0,$  and is \emph{asymptotic} if $\displaystyle{\lim_{ n\to \infty}} ||T^n(x)-T^n(y)||=0.$ The proximal(asymptotic) pairs in $X \times X$ define the \emph{proximal(asymptotic)  relations $P$($A$)} on $X$.

 \begin{Definition}
	A family $G$ of linear operators defined on a topological vector space $X$ is called \emph{equicontinuous} if for every $\varepsilon>0$ there exists $\delta>0$ such that $||x-y||< \delta$ implies that $||T(x)-T(y)||<\varepsilon$ for all $T\in G.$\par 
	A linear dynamical system $(X,T)$ is called \emph{equicontinuous} if the family $\{T^n:n\in \mathbb{N}\}$ is equicontinuous.
\end{Definition}

We can equivalently consider equicontinuity at the origin in $X$. $(X,T)$ is \emph{equicontinuous at the origin} if  for every $\varepsilon>0$ there exists $\delta>0$ such that $||x||< \delta$ implies that $||T^n(x)||<\varepsilon$ for all $n\in \mathbb{N}$. It is not difficult to see that  $(X,T)$ is equicontinuous at the origin if and only if  $(X,T)$ is equicontinuous.

\begin{Definition}
	The linear  system $(X,T)$ is called \emph{sensitive} if there exists $\delta>0$ such that for every $x\in X$ and every $\varepsilon>0$ there exist a point $y\in X$ with $||x-y||<\varepsilon$ and $n\in \mathbb{N}$ for which $||T^n(x)-T^n(y)||\geq\delta.$ 
	
	This $\delta >0$ is called the \emph{sensitivity constant} for $T$.
\end{Definition}

We can equivalently consider sensitivity at the origin in $X$. $(X,T)$ is \emph{sensitive at the origin with sensitivity constant $\delta > 0$} if   there exists $\delta>0$ such that for every $\varepsilon>0$ there exists $x \in X$ with  $||x||< \varepsilon$ and $n\in \mathbb{N}$ for which $||T^n(x)|| \geq \delta$. It is not difficult to see that  $(X,T)$ is sensitive at the origin with sensitivity constant $\delta > 0$ if and only if  $(X,T)$ is sensitive with sensitivity constant $\delta > 0$.

\begin{Lemma} [Proposition 6.1 \cite{GS}]
	Every hypercyclic operator on an $ F $-space has sensitive dependence on
	initial conditions.
\end{Lemma}

\begin{Theorem}\label{minimal}
	A linear dynamical system is either sensitive or equicontinuous.
\end{Theorem}
\begin{proof} We only need to check the properties at the origin.
	Suppose that a linear operator $T$ is not sensitive at $0$. Then there exists $x\in X$ such that for every $\varepsilon>0$ there exists $\eta >0 $ such that $||x||< \eta \implies ||T^n(x)||< \varepsilon$ for every $n\in \mathbb{N}.$  Hence $(X,T)$ is equicontinuous at the origin.
\end{proof}

\begin{Remark}
	We know that a transitive operator is sensitive. Hence from the above theorem, we can say that a transitive operator can never be equicontinuous.
\end{Remark} 

\bigskip
	
We now consider families:	
	\begin{Definition} \cite{akin}
		A non-empty collection $\mathcal{F}$ of subsets of $\mathbb{N}$ is called a \emph{Furstenberg family} (or simply \emph{family}) on $\mathbb{N}$ if every member $A\in \mathcal{F}$ is infinite and $\mathcal{F}$ is hereditarily upward i.e., if $A\in \mathcal{F},$ and $A\subset B$ then $B\in \mathcal{F}$.\par
		 Moreover, $\mathcal{F}$ is a \emph{filter} if for every pair of members $A,B \in \mathcal{F},$ $A\cap B \in \mathcal{F}.$ For a family $\mathcal{F}$ its \emph{dual} $k\mathcal{F}$ is defined as $k\mathcal{F}=\{A: A\cap B \neq \emptyset \ \text{for every}\ B\in \mathcal{F}\}.$ Let $\mathcal{J}$ denote the set of all infinite subsets of $\mathbb{N}$ and so $k\mathcal{J}$ is the set of all cofinite subsets of $\mathbb{N}$.
		\end{Definition}
		
		A family $\mathcal{F}$ is called \emph{shift invariant} if for every $i\in \mathbb{N}$ and $A\in \mathcal{F}$ both the sets $(A+i)\cap \mathbb{N}$ and $(A-i)\cap \mathbb{N}$ belong to $\mathcal{F}.$\par 
        Given a family $\mathcal{F}$ on $\mathbb{N}$, we consider another family $$\widetilde{\mathcal{F}}=\{A\subset \mathbb{N}: \forall\ N \in \mathbb{N}\ \exists\ B\in \mathcal{F}\ \text{such that}\ (B+[-N,N])\cap \mathbb{N} \subset A\}$$ 
		Clearly $\widetilde{\mathcal{F}} \subset \mathcal{F}$ and $\widetilde{\mathcal{F}}$ is shift invariant.\par  
		
		We say that a sequence $(x_n)$ is \emph{$\mathcal{F}$-convergent} to $x$ or $\mathcal{F}$-$\lim_n x_n =x$ if for every open set $U$ containing $x,$ the set $\{n\in \mathbb{N}: x_n\in U\}\in \mathcal{F}$ \cite{F1}.

		Recall that $(X,T)$ is called \emph{$\mathcal{F}$-transitive} if for every pair of opene sets $U,V \subset X$ the set $N(U,V)\in \mathcal{F}$ (c.f. \cite{F1}) and 
 $(X,T)$ is called \emph{hereditarily $\mathcal{F}$-transitive} if for every pair of opene sets $U,V \subset X$ and every $A\in \mathcal{F},$ $N(U,V)\cap A \in \mathcal{F}$ (c.f. \cite{bes ST})

    \begin{Lemma} [Lemma 2.3 \cite{bes ST}] \label{F WEAK} For the linear system $(X,T)$ let $\mathcal{F}$ be a family  on $\mathbb{N}$. Then $(X,T)$ is $\widetilde{\mathcal{F}}$-transitive if and only if $(X,T)$ is $\mathcal{F}$-transitive and weakly mixing.
	\end{Lemma}
	
\bigskip	
	
\section{On $\mathcal{F}$-Transitive Operators}

	We observe some basic properties of transitive operators:

	\begin{Lemma}\label{lemma0}
	 Let $X$ be an infinite dimensional, separable $F$-space and $T\in \mathcal{L}(X)$ be topologically transitive. Then for every $x\in X$ and $n\in \mathbb{N}$, the set $T^{-n}(x)$ has empty interior.
	\end{Lemma}
    \begin{proof}
        Suppose to the contrary that there exist some opene set $U$ and $x\in X$ such that $T^n(U)=\{x\}$ for some $n\in \mathbb{Z}_+.$ Let $y$ be a transitive point. Then there exists $m\in \mathbb{N}$ such that $T^m(y)\in U$ and so $T^{m+n}(y)=x.$\par 
        Case $(i)$: Suppose that $x$ is a transitive point. Then there exits $k\in \mathbb{N}$ such that $T^k(x)\in U$ and so $T^{k+n}(x)=x,$ which is a contradiction.\par 
        Case $(ii)$: Suppose that $x$ is not a transitive point. Then there exists an opene set $V\subset X$ such that $\mathcal{O}(x)\cap V = \emptyset.$ Take $V^{\prime}=V\setminus \{y,\dots, T^{m+n}(y)\}.$ Clearly, $V^{\prime}$ is an opene set and $\mathcal{O}(y)\cap V^{\prime}=\emptyset,$ which is again a contradiction. 
        
        This completes the proof.
    \end{proof}
	\begin{Lemma}\label{lemma1}
		Let $X$ be an infinite dimensional, separable $F$-space and $T\in \mathcal{L}(X)$. Suppose that $T$ is topologically transitive. Then there exist two transitive points $x,y$  such that $\mathcal{O}(x)\cap \mathcal{O}(y)=\emptyset.$
	\end{Lemma}
\begin{proof}
	Let $x$ be a transitive point for $(X,T)$ and $n\in \mathbb{Z}_+.$ Clearly $X\setminus \{T^{n}(x)\}$ is an open and dense subset of $X.$ On the other hand $X\setminus T^{-n}(x)$ is also open in $X.$ Now we claim that $X\setminus T^{-n}(x)$ is dense in $X.$
	 Suppose to the contrary that $X\setminus T^{-n}(x)$ is not dense in $X.$ Then there exists an opene set $U\subset X$ such that $(X\setminus T^{-n}(x))\cap U =\emptyset,$ which implies that $T^n(U)=x,$ which gives a contradiction from Lemma \ref{lemma0}. 
    Then from Baire Category theorem we can conclude that there is a transitive point $y\in X$ such that $y\notin \mathcal{O}(x) \bigcup [ \displaystyle{\bigcup_{n\in \mathbb{N}}}T^{-n}(x)].$ \par 
	 Now if $\mathcal{O}(x)\cap \mathcal{O}(y)\neq\emptyset$, then there exists $i, j\in \mathbb{N}$ such that $T^i(y)=T^j(x) \Rightarrow y\in T^{j-i}(x),$ which is a contradiction. This completes the proof. 
\end{proof}

We now define a general form of Hypercyclicity Criterion using a family $\mathcal{F}$ on $\mathbb{N}$ to characterize $\mathcal{F}$-transitivity properties for linear operators. 
\begin{Definition}\label{def1}
	Let $X$ be  an infinite dimensional, separable $F$-space, $T\in \mathcal{L}(X)$  and $\mathcal{F}$ be a family on $\mathbb{N}.$ We say that $T$ satisfies \emph{$\mathcal{F}$-Transitivity Criterion} if there exist two dense subsets $D_1,$ $D_2$ of $X$ and two sequence of mappings $I_n:D_1\to X,$ $S_n:D_2\to X$ such that\par 
	$(i)$ $\mathcal{F}$-$\lim_n (I_n(x), T^n(I_n(x)))=(x,0)$ for every $x\in D_1$\par 
	$(ii)$ $\mathcal{F}$-$\lim_n (S_n(y), T^n(S_n(y)))=(0,y)$ for every $y\in D_2$.
\end{Definition}

\begin{Theorem}\label{hyper}
	Let $X$ be  an infinite dimensional, separable $F$-space, $T\in \mathcal{L}(X)$  and $\mathcal{F}$ be a family on $\mathbb{N}.$ Then there exists some filter $\mathcal{F}^{\prime}\subset \mathcal{F}$ so that the following conditions are equivalent:
	\begin{enumerate}
		\item $(X,T)$ is $\widetilde{\mathcal{F}}$-transitive.\par 
		\item $T$ satisfies $\mathcal{F}^{\prime}$-Transitivity Criterion. \par 
		\item $(X,T)$ is hereditarily $\mathcal{F}^{\prime}$-transitive.
	\end{enumerate}
\end{Theorem}

\begin{proof}
	$(1)\Rightarrow (2)$ Suppose that $(X,T)$ is $\widetilde{\mathcal{F}}$-transitive. Then from Lemma \ref{F WEAK}, $T$ is weakly mixing and $\mathcal{F}$-transitive. Consequently, from Lemma \ref{lemma1}, there exist two transitive points $x,y\in X$ such that $\mathcal{O}(x)\cap \mathcal{O}(y)= \emptyset.$ Let $(U_k), (V_k)$ and $(W_k)$ be the countable decreasing neighbourhood base at $x, y$ and $0$ respectively. For every $k\in \mathbb{N},$ we take \par 
	\centerline
	{$F_k=N(U_k, W_k)\cap N(W_k,V_k).$ } 

Since $T$ is $\widetilde{\mathcal{F}}$-transitive, so by using Lemma \ref{F WEAK} and \ref{F1}, we can conclude that $F_k\in \mathcal{F}$. Moreover, the subfamily $\mathcal{F}^{\prime}=\{F\subset \mathbb{N}: F_k\subset F \ \text{for some}\ k\in \mathbb{N}\}$ forms a filter as $F_{k+1}\subset F_k$ for every $k\in \mathbb{N}.$ \par
Also we  observe that $\displaystyle{\bigcap_{k=1}^{\infty}}F_k = \emptyset.$ If there is some $m\in \mathbb{N}$ with $m\in \displaystyle{\bigcap_{k=1}^{\infty}}F_k,$ then for every $k\in \mathbb{N},$ there exists $z_k\in W_k$ such that $T^m(z_k)\in V_k.$ Then we get a sequence $(z_k)\to 0$ but $T^m(z_k)\to y \neq 0,$ which contradicts continuity of $T^m.$\par 
\vspace{.1cm}
Let  $E_k=F_k\setminus F_{k+1}.$ Then clearly $E_i \cap E_j = \emptyset$ when $i\neq j$. We claim that for infinitely many $k,$ $E_k \neq \emptyset.$ On the contrary suppose that only finitely many $E_k$'s are non-empty. Then there exists some $k_0\in \mathbb{N}$ such that $E_i=\emptyset$ for all $i\geq k_0 \Rightarrow F_i=F_{k_0}$ for all $i\geq k_0.$ But this contradicts the fact that $\displaystyle{\bigcap_{k=1}^{\infty}}F_k = \emptyset.$ Also for every $k,$ $F_k=\displaystyle{\bigcup_{i\geq k}}E_i$. Indeed $\displaystyle{\bigcup_{i\geq k}}E_i \subset F_k$ and  let $m\in F_k.$ Since $\displaystyle{\bigcap_{k=1}^{\infty}}F_k = \emptyset$, there exists some $i\geq k$ such that $m\in F_i$ but $m\notin F_{i+1},$ which implies $m\in E_i.$ So without any loss of generality we can assume that: \par 
\vspace{.15cm}
$(i)$ For every $k\in \mathbb{N},$ $E_k \neq \emptyset.$\par 
$(ii)$ $E_i \cap E_j = \emptyset$ when $i\neq j.$\par 
$(iii)$ For every $k\in \mathbb{N}, F_k=\displaystyle{\bigcup_{i\geq k}}E_i.$ \par 
Let $k\in \mathbb{N}.$ Then for every $m\in E_k$, there exist some $u_{(k,m)}\in U_k,$ $w_{(k,m)}\in W_k$ such that $T^m(u_{(k,m)})\in W_k$ and $T^m(w_{(k,m)})\in V_k$. Let  $D_1=\{T^jx:j\geq 0\}$ and $D_2=\{T^jy:j\geq 0\}$. Also, define two sequences of maps $I_n:D_1\to X$ and $S_n:D_2\to X$ in such a way that for every $j\in \mathbb{N}$
\begin{center}
	
	$I_{n}(T^jx)=
	\begin{cases}
		T^ju_{(k,n)}& \text{when}\ n\in E_k\  \text{for some}\  k\\ 
		0& \text{when}\ n\notin \displaystyle{\bigcup_{k\in \mathbb{N}}}E_k
	\end{cases}$
\end{center}

\begin{center}
		$S_{n}(T^jy)=
	\begin{cases}
		T^jw_{(k,n)}& \text{when}\ n\in E_k\  \text{for some}\  k\\ 
		0& \text{when}\ n\notin \displaystyle{\bigcup_{k\in \mathbb{N}}}E_k
	\end{cases}$
\end{center}
Now we  show that for every $j\in \mathbb{N},$ $\mathcal{F}^{\prime}-\displaystyle{\lim_n}(I_n(T^jx),T^n(I_n(T^jx)))=(T^jx, 0).$ 

Let $U$ and $W$ be two open sets containing $T^jx$ and $0$ respectively. Then there exists $k\in \mathbb{N},$ such that $x\in U_k \subset T^{-j}(U)$ and $0\in W_k \subset T^{-j}(W).$ Now for every $m\in F_k$ there is some $i\geq k$ with $m\in E_i.$ Then we have
\begin{center}
	$u_{(i,m)}\in U_i\subset U_k\subset T^{-j}(U)$ and $T^m(u_{(i,m)})\in W_i\subset W_k\subset T^{-j}(W)$\\
	$\Longrightarrow T^j(u_{(i,m)})\in U$ and $T^j(T^m(u_{(i,m)}))\in W$\\
	$\Longrightarrow I_m(T^jx)\in U$ and $T^m(I_m(T^jx))\in W$
	\end{center}
	Hence $F_k\subset \{n\in \mathbb{N}: (I_n(T^jx),T^n(I_n(T^jx))) \in U \times W\}\in \mathcal{F}^{\prime}.$ Consequently, $\mathcal{F}^{\prime}-\displaystyle{\lim_n}(I_n(T^jx),T^n(I_n(T^jx)))=(T^jx, 0).$ Similarly for every $j\in \mathbb{N},$ one can prove that $\mathcal{F}^{\prime}-\displaystyle{\lim_n}(S_n(T^jy),T^n(S_n(T^jy)))=(0,T^jy).$ 
	
	Thus $T$ satisfies $\mathcal{F}^{\prime}$-Transitivity Criterion.\par 
	$(2)\Rightarrow (3)$ Let $U$ and $V$ be two opene subsets of $X$. We  fix  two other opene subsets $U^{\prime}, V^{\prime}$ and a neighbourhood $W$ of $0$ such that $U^{\prime}+W \subset U$ and $V^{\prime}+W \subset V.$ Choose $x\in U^{\prime}\cap D_1$ and $y\in V^{\prime}\cap D_2.$ Then from the conditions of $\mathcal{F}^{\prime}$-Transitivity Criterion, we get $F\in \mathcal{F}^{\prime}$ such that for every $n\in F$ $I_n(x)+S_n(y)\in U^{\prime}+W$ and $T^n(I_n(x)+S_n(y))\in W+V^{\prime}.$ This implies that $F\subset N(U^{\prime}+W, V^{\prime}+W)$ and so $N(U,V)\in \mathcal{F}^{\prime}.$ Now since $\mathcal{F}^{\prime}$ is filter, it is clear that $T$ is hereditarily $\mathcal{F}^{\prime}$-transitive. \par 
	$(3)\Rightarrow (1)$ Let $U_1,V_1,U_2,V_2$ be opene subsets of $X.$ Clearly $N(U_1,V_1)\cap N(U_2,V_2)\in \mathcal{F}.$ So $T$ is weakly mixing and $\mathcal{F}$-transitive. Hence $T$ is $\widetilde{\mathcal{F}}$-transitive. 
\end{proof}
\begin{Corollary}\label{hyper1}
	Let $X$ be an infinite dimensional, separable $F$-space and $T\in \mathcal{L}(X).$ Also, let $\mathcal{F}$ be a filter on $\mathbb{N}.$ Then the following conditions are equivalent.
	\begin{enumerate}
		\item $(X,T)$ is $\mathcal{F}$-transitive.\par
		\item $T$ satisfies $\mathcal{F}$-Transitivity Criterion.\par
		\item $(X,T)$ is hereditarily $\mathcal{F}$-transitive.
	\end{enumerate}
\end{Corollary} 
\begin{Corollary}
Let $X$ be an infinite dimensional, separable $F$-space and $T\in \mathcal{L}(X).$ Then the following conditions are equivalent.\par 
\begin{enumerate}
    \item $(X,T)$ is weakly mixing.
    \item There exists some sequence $(n_k)$ of positive integers such that $T$ satisfies Hypercyclicity Criterion with respect to $(n_k).$
    \item $T$ satisfies $\mathcal{F}$-Transitivity Criterion, where the filter $\mathcal{F}=\{A:\ \exists \ k_0\in \mathbb{N}\  \text{such that}\  \{n_k:k\geq k_0\}\subset A\}$. 
   
\end{enumerate}  
Hence for this filter $\mathcal{F},$ $\mathcal{F}$-Transitivity Criterion is equivalent to Hypercyclicity Criterion.
\end{Corollary}
Now from Theorem \ref{hyper}, we can easily characterize mixing operators by using $k\mathcal{J}$-Transitivity Criterion.
\begin{Corollary}\label{mixing}
	Let $X$ be an infinite dimensional, separable $F$-space and $T\in \mathcal{L}(X).$ Then $(X,T)$ is mixing if and only if there exist two dense subsets $D_1,$ $D_2$ of $X$ and two sequence of mappings $I_n:D_1\to X,$ $S_n:D_2\to X$ such that the following conditions hold:\par 
	$(i)$ $\lim_n(I_n(x), T^n(I_n(x)))=(x,0)$ for every $x\in D_1$\par 
	$(ii)$ $\lim_n (S_n(y), T^n(S_n(y)))=(0,y)$ for every $y\in D_2$.
\end{Corollary}

\begin{Remark}\label{kernel}
    From the above Corollary one can see that if $T$ has dense generalized Kernel (i.e., the set $\displaystyle{\bigcup_{n\geq 1}}Ker(T^n)$ is dense in $X$), then $(X,T)$ is mixing if and only if $T$ satisfies Kitai's Criterion. Hence a weighted backward shift operator is mixing if and only if it satisfies Kitai's Criterion.
\end{Remark}
In [ Theorem 2.4 \cite{bes ST}], $\mathcal{F}$-Transitivity Criterion is defined in the following form:
		Let $\mathcal{F}$ be a family on $\mathbb{N}.$ Then a bounded linear operator $T$ defined on an infinite dimensional, separable $F$-space $X$ satisfies $\mathcal{F}$-Transitivity Criterion if there exist two dense subsets $D_1,$ $D_2$ of $X$ and a sequence of mapping $S_n:D_2\to X$ such that\par 
		$(i)$ $\mathcal{F}$-$\lim_n T^n(x)=0$ for every $x\in D_1$\par 
		$(ii)$ $\mathcal{F}$-$\lim_n (S_n(y), T^n(S_n(y)))=(0,y)$ for every $y\in D_2$.

It is clear that for any family $\mathcal{F},$ $\mathcal{F}$-Transitivity Criterion defined in \cite{bes ST} implies $\mathcal{F}$-Transitivity Criterion of Definition \ref{def1}. \par 
On the other hand, in the Definition of \cite{bes ST}, if we take $\mathcal{F} =k\mathcal{J}$, then it is equivalent to Kitai's Criterion. From \cite{mixing example} it is clear that mixing condition is not equivalent to Kitai's Criterion. But from Corollary \ref{mixing}, we can say that for $\mathcal{F} =k\mathcal{J}$ our definition is equivalent to mixing. Hence our definition of $\mathcal{F}$-Transitivity Criterion is strictly weaker than the Criterion mentioned in \cite{bes ST}. In this article we call the $\mathcal{F}$-Transitivity Criterion of \cite{bes ST} as $\mathcal{F}$-Kitai's Criterion.

\bigskip
In \cite{prox1,prox2} proximal and asymptotic relations are studied  using Furstenberg families. We  define these  relations in linear structure and examine their connections with $\mathcal{F}$-tansitive operators. 

\begin{Definition}
    Let $X$ be an infinite dimensional, separable $F$-space and $T\in \mathcal{L}(X).$ Also, let $S$ be an infinite subset of $\mathbb{N}$ and $\mathcal{F}$ be a family on $\mathbb{N}$. Then a pair $(x,y)\in X\times X$ is called 
    \begin{enumerate}
       \item
        \emph{$S$-asymptotic} if $\displaystyle{\lim_{n\in S, n\to \infty}} ||T^n(x)-T^n(y)||=0.$ 
        \item \emph{$\mathcal{F}$-proximal} if for every $\varepsilon>0$ the set $\{n\in \mathbb{N}: ||T^n(x)-T^n(y)||<\varepsilon\} \in \mathcal{F}.$  
    \end{enumerate}

    The set of all $S$-asymptotic pairs $$A_S=\{(x,y):(x,y) \ \text{is $S$-asymptotic}\}$$ forms a equivalence relation on $X\times X$ and it is called \emph{$S$-asymptotic relation}. Moreover, for every $x\in X,$ the set $A_S(x)=\{y:(x,y) \ \text{is $S$-asymptotic}\}$ is called \emph{$S$-asymptotic cell of $x.$} 

    The set of all $\mathcal{F}$-proximal pairs is  $$P_{\mathcal{F}}=\{(x,y):(x,y) \ \text{is $\mathcal{F}$-proximal}\}$$ forms a reflexive, symmetric relation on $X \times X$ and for every $x\in X,$ the set $P_{\mathcal{F}}(x)=\{y:(x,y) \ \text{is $\mathcal{F}$-proximal}\}$ is called \emph{$\mathcal{F}$-proximal cell} of $x.$ 
        \end{Definition}
   Note that  when $S=\mathbb{N},$ then $(x,y)$ is a asymptotic pair and when $\mathcal{F}= \mathcal{J},$ then $(x,y)$ is a proximal pair. It is clear that a pair $(x,y)$ is asymptotic if and only if it is $k\mathcal{J}$-proximal. 

    \begin{Theorem}
     Let $X$ be an infinite dimensional, separable $F$-space and $T\in \mathcal{L}(X).$ Suppose that $(X,T)$ is transitive. Then there exists an infinite set $S\subset \mathbb{N}$ such that for every $x\in X$ the set $A_S(x)$ is dense in $X.$
    \end{Theorem}
    \begin{proof}
        Now $(X,T)$ is transitive and let $x\in X$. Then there exist a transitive point $z\in X$ and an increasing sequence $(n_k)$ of natural numbers such that $T^{n_k}(z)\to 0.$ Let  $S=\{n_k:k\in \mathbb{N}\}$, then $S$ is infinite. Let   $U$ be any opene subset of $X.$ Then there exists some $j\in \mathbb{N}$ such that $T^j(z)\in U-x = \{u - x: u \in U\}.$ Take $y=x+T^j(z)\in U$. Now $||T^{n_k}(y)-T^{n_k}(x)||=||T^{n_k}(T^j(z))||=||T^j(T^{n_k}(z))||\to 0$ as $k\to \infty.$ So $y\in A_S(x) \cap U$ and it proves that $A_S(x)$ is dense in $X.$
  \end{proof}

\begin{Theorem}
	Let $X$ be an infinite dimensional, separable $F$-space and $T\in \mathcal{L}(X).$ Then $(X,T)$ is weakly mixing if and only if there exists a filter $\mathcal{F}$ such that the following conditions hold \par
	\begin{enumerate}
		\item For every $S\in k\mathcal{F}$ and $x\in X$, the set $A_S(x)$ is dense in $X;$
		\item For every opene  $U\subset X$ and for every open  $W \ni 0$ the set $N(W,U)\in \mathcal{F}.$
	\end{enumerate}
\end{Theorem}
\begin{proof}
	Suppose that $T$ is weakly mixing then $T$ satisfies Hypercyclicity Criterion with respect to some sequence $(n_k)$. Consider the filter $\mathcal{F}=\{A\subset \mathbb{N}: \ \text{there exists}\ k_0\in \mathbb{N}\ \text{with}\ \{n_k:k\geq k_0\}\subset A\}.$ Let $S\in k\mathcal{F}, \ x\in X$ and $U \subset X$ be  opene.  Then from the first condition of Hypercyclicity Criterion, there exists $z$ with $x+z\in U$ such that $T^{n_k}(z)\to 0.$ Consequently, $||T^{n_k}(x+z)-T^{n_k}(x)||\to 0$ and so $A_S(x)$ is dense in $X.$ Futhermore, let $W$ be an open neighbourhood of $0$ and $V$ be an opene subset of $X$. Also, let $y\in V$. Then using the second condition of hypercyclicity criterion there exists a sequence $(y_k)$ such that $(y_k)\to 0$ and $T^{n_k}(y_k)\to y.$ Therefore for sufficiently large $k$, $y_k\in W$ and $T^{n_k}(y_k)\in V$. Hence $N(W,V)\in \mathcal{F}$.\par
	Conversely, suppose that there is a filter $\mathcal{F},$ for which the given conditions hold. Let $U$ and $V$ two opene subsets of $X.$ Then there exists a neighbourhood $W$ of $0$ and two opene subsets $U^{\prime}, V^{\prime}$ such that $U^{\prime}+W \subset U$ and $V^{\prime}+W \subset V.$ From the given conditions the set $S=N(W,V^{\prime}) \in \mathcal{F}.$ Then we can write $S=\{n_1 < n_2 < \dots \}$. Clearly $S\in k\mathcal{F}.$ Now from the asymptotic condition there is some $z\in U^{\prime}$ such that $||T^{n_k}(z)-T^{n_k}(0)||=||T^{n_k}(z)||\to 0.$ Then there exists $k_0\in \mathbb{N}$ such that for every $k\geq k_0$ $T^{n_k}(z)\in W$ and $T^{n_k}(w_k)\in V^{\prime}$ for some $w_k\in W.$ Consequently for every $k\geq k_0$ $z+w_k\in U$ and $T^{n_k}(z+w_k)\in V.$ Hence $\{n_k:k\geq k_0\}\subset N(U,V)$ and so $N(U,V)\in \mathcal{F}.$ This implies that $T$ is weakly mixing.
\end{proof}
\begin{Theorem}\label{kitai prox}
	Let $X$ be an infinite dimensional, separable $F$-space and $T\in \mathcal{L}(X).$ Also, let $\mathcal{F}$ be a filter on $\mathbb{N}.$ Then $T$ satisfies $\mathcal{F}$-Kitai's Criterion if and only if $(X,T)$ is $\mathcal{F}$-transitive and for every $x\in X,$ the set $P_{\mathcal{F}}(x)$ is dense in $X$.  
\end{Theorem}

\begin{proof}
	Suppose that $T$ satisfies $\mathcal{F}$-Kitai's Criterion. Consequently, $T$ satisfies $\mathcal{F}$-Transitivity Criterion. Then from Corollary \ref{hyper1}, $(X,T)$ is $\mathcal{F}$-transitive. Also, from the first condition of $\mathcal{F}$-Kitai's Criterion there exists a dense subset $D$ such that $\mathcal{F}$-$\lim_n T^n(y)=0$ for every $y\in D.$ Let $x\in X$ and $U$ be any opene subset of $X.$ Take $z\in (U-x)\cap D$. Then $\mathcal{F}$-$\lim_n T^n(z)=0$. Consequently, for every $\varepsilon>0$, the set  $\{n\in \mathbb{N}:||T^{n_k}(x+z)-T^{n_k}(x)||<\varepsilon\}\in \mathcal{F}.$ So the pair $(x, x+z)$ is $\mathcal{F}$-proximal. Hence the set $P_{\mathcal{F}}(x)$ is dense in $X$.\par 
	Conversely, suppose that $(X,T)$ is $\mathcal{F}$-transitive. Then from Corollary \ref{hyper1}, $T$ satisfies $\mathcal{F}$-transitivity Criterion. Consequently, $T$ satisfies the second condition of $\mathcal{F}$-Kitai's Criterion. Now from the given condition the set $P_{\mathcal{F}}(0)$ is dense in $X.$ Then $\mathcal{F}$-$\lim_n T^n(x)=0$ for every $x\in P_{\mathcal{F}}(0).$ Hence $T$ satisfies $\mathcal{F}$-Kitai's Criterion.
\end{proof}

Now from the above Theorem we can say that for a filter $\mathcal{F}$, $\mathcal{F}$-transitivity along with the condition that for every point $x\in X$ the corresponding proximal cell $P_{\mathcal{F}}(x)$ is dense in $X$, is equivalent with $\mathcal{F}$-Kitai's Criterion.   Hence proximity relation fills the gap between $\mathcal{F}$-transitivity and $\mathcal{F}$-Kitai's Criterion. In particularly, by taking $\mathcal{F}= K\mathcal{J}$ one can say that the asymptotic relation fills the gap between mixing operators and Kitai's Criterion. We note this  in the following Corollary.
\begin{Corollary}
	Let $T$ be a bounded linear operator defined on an infinite dimensional, separable $F$ space. Then $T$ satisfies Kitai's Criterion if and only if $(X,T)$ is mixing and for every $x\in X,$ the set $A_{\mathbb{N}}(x)$ is dense in $X$.  
\end{Corollary}
\begin{proof}
	By taking $\mathcal{F}= K\mathcal{J}$ in Theorem \ref{kitai prox}, one can get the result.
\end{proof}

\bigskip 

We now define a weakening of the proximal relation for linear operators:
\begin{Definition}
	Let $X$ be an infinite dimensional, separable $F$-space and $T\in \mathcal{L}(X).$ Then  $(x,y)\in X\times X$ is called a \emph{regionally proximal pair} if there exist two sequences $(x_n)\to x$ and $(y_n)\to y$ such that $\displaystyle{\liminf_{n\to \infty}} ||T^n(x_n)-T^n(y_n)||=0.$    
	Then $$RP=\{(x,y):(x,y) \ \text{is regionally proximal}\}$$ is called the \emph{regionally proximal relation} on $X$.
\end{Definition}
\begin{Theorem}\label{RP}	
	Let $X$ be an infinite dimensional, separable $F$-space and $T\in \mathcal{L}(X).$ Suppose that $(X,T)$ is transitive. Then the regionally proximal relation $RP=X\times X.$ 
\end{Theorem}
\begin{proof}
	Let $(x,y)\in X\times X.$ Also, let $(U_k)$ and $(W_k)$ be the neighbourhood base of $y$ and $0$ respectively. 
	
	Now choose $n_k\in N(U_k-x, W_k)$. Then for every $k$ there is some $z_k\in U_k-x$ such that $T^{n_k}(z_k)\in W_k.$ So the sequence $(x+z_k)$ converges to $y$ and $||T^{n_k}(x+z_k)-T^{n_k}(x)||\to 0$. Hence $(x,y)\in RP.$ 
\end{proof}

\bigskip

\section{Strongly Transitive Operators}

Again recall that $(\widehat{X},\widehat{T})$ is a compactification of $(X,T)$ where $X$ is an infinite dimensional, seperable Banach space.
 
 \bigskip 
 
 We now recall another stronger form of transitivity from \cite{AN} and define that for linear systems.  The 	 system $(\widehat{X},\widehat{T})$  is \emph{strongly transitive} if for every opene $U \subset \widehat{X}$, $\bigcup \limits_{n=1}^\infty {\widehat{T}}^n(U) = \widehat{X}$, equivalently for every  opene set $U \subset \widehat{X}$ and every point $x \in \widehat{X}$,	 the set $N_{\widehat{T}}(U,x)$ is nonempty. It can be noted that for invertible systems, strongly transitive is equivalent to minimality, and so linear systems can be strongly transitive only in the  non-invertible cases. The 	 system $(\widehat{X},\widehat{T})$  is \emph{strongly product transitive} if the product system $(\underbrace{\widehat{X} \times \ldots \times  \widehat{X}}_n,\underbrace{\widehat{T} \times \ldots \times \widehat{T}}_n)$  is strongly transitive for all $n \in \mathbb{N}$.
 
 As discussed earlier, we only use open sets in this definition which does not require any compactification and so can be defined for any linear system on an infinite dimensional, seperable $F$-space.  A linear dynamical system $(X,T)$ is \emph{Strongly Transitive (ST)} if for every opene $U \subseteq X$, $\bigcup \limits_{n=1}^\infty \ T^n(U) = X$; and is  \emph{strongly product transitive (SPT)} if the product system $(\underbrace{X \times \ldots \times  X}_n,\underbrace{T \times \ldots \times T}_n)$  is strongly transitive for all $n \in \mathbb{N}$.	  

	\begin{Theorem}\label{ST}
		Let $X$ be an infinite dimensional, separable $F$-space and $T\in \mathcal{L}(X).$ Then $(X,T)$ is strongly transitive if and only if there exists a dense set $D\subset X$ satisfying:
		
		$(i)$ For every $x\in D$, there exists $n\in \mathbb{N}$ such that $T^n(x)=0.$\par 
		$(ii)$ For every $y\in X$ there exists a sequence $(n_k)$ in $\mathbb{N}$ and a point $y_{k}\in T^{-n_k}(y)$ such that $y_{k}\to 0$.
	\end{Theorem}
	\begin{proof}
		Suppose that $(X,T)$ is strongly transitive. Then for every opene $U\subset X = \bigcup \limits_{n=1}^\infty \ T^n(U),$ there exist $x\in U$ and $n\in \mathbb{N}$ such that $T^n(x)=0.$ Hence the set $D=\{x:\  T^n(x)=0\ \text{for some}\ n\in \mathbb{N}\}$ is dense in $X$ and this proves condition $(i).$ \par 
		Let $(U_k)$ be the neighbourhood base at $0$ and $y\in X.$ Then for every $k\in \mathbb{N}$, there exist $y_k\in U_k$ and $n_{k}\in \mathbb{N}$ such that $T^{n_k}(y_k)=y.$ This proves condition $(ii).$\par 
		 
		Conversely, suppose $T$ satisfies the given conditions. Let $U$ be opene in $X$ and $y\in X.$ Choose $x\in U\cap D.$ Then from condition $(i)$, there exists $N\in \mathbb{N}$ such that $T^n(x)=0$ for every $n\geq N.$ On the other hand from condition $(ii)$, there exists a sequence $(n_k)$ and $y_{k}\in T^{-n_k}(y)$ for every $k\in \mathbb{N}$ such that $y_{k}\to 0$. Consequently, for large enough $k,$ $x+y_k\in U$ and $T^{n_k}(x+y_k)=y.$ Hence $X = \bigcup \limits_{n=1}^\infty \ T^n(U)$ and so $(X,T)$ is strongly transitive.
	\end{proof}

We recall:

\begin{Theorem} [Theorem 8 \cite{racsum}] \label{racs}
	Let $T$ be a topologically transitive operator. If there exists a dense subset $D$ such that for every $x\in D$, the orbit of $x$ is bounded, then $T$ is weakly mixing.
\end{Theorem}

\begin{Corollary}
	A strongly transitive operator is weakly mixing.
\end{Corollary}
\begin{proof}
	The result follows from Theorem \ref{racs} and Theorem \ref{ST} since the  pre-images of $0$ is dense.
\end{proof}

    \begin{Remark} From Theorem \ref{ST} and Remark \ref{kernel}, we can say that a strongly transitive operator is mixing if and only if it satisfies Kitai's Criterion. So mixing operators which do not satisfy Kitai's Criterion, cannot be strongly transitive.

 In case of topological dynamics, strongly transitive is independent of mixing or weakly mixing \cite{AN}. However,  in case of linear dynamics, strongly transitivity implies weakly mixing. We leave this question open:

    \textbf{Question:} Does Strongly transitive imply mixing in linear dynamics? \end{Remark}

\bigskip

\begin{Theorem}\label{SPT}
	Let $X$ be an infinite dimensional separable $F$-space and $T\in \mathcal{L}(X).$ Then $(X,T)$ is strongly product transitive if and only if there exists a dense set $D\subset X$ such that the following conditions hold:\par
	$(i)$ For every $x\in D$, there exists $n\in \mathbb{N}$ such that $T^n(x)=0.$\par
	$(ii)$ For any finite set $\{y_1,\dots,y_n\}\subset X$, there exists a sequence $(n_k)$ such that for every $i\in \{1,\dots,n\}$, we get a sequence of points $(y^i_k)_k$ such that $y^i_{k}\in T^{-n_k}(y_i)$ and $y^i_{k}\to 0$ as $k \to \infty.$
\end{Theorem}
\begin{proof}
	
	Suppose that $(X,T)$ is strongly product transitive. Then for every opene $U\subset X,$ there exist $x\in U$ and $n\in \mathbb{N}$ such that $T^n(x)=0.$ Hence the set $D=\{x:\  T^n(x)=0\ \text{for some}\ n\in \mathbb{N}\}$ is dense in $X$ and this proves condition $(i).$ \par
	Let $(U_k)$ be the neighbourhood base at $0$ and $\{y_1,\dots,y_n\}$ be any finite subset $X.$ Since $X$ is strongly product transitive,
	$(
	\underbrace{X \times \cdots \times X}_{n},
	\underbrace{T \times \cdots \times T}_{n})$ is strongly transitive. Then for every $k\in \mathbb{N}$, there exist $(y^1_k,\cdots,y^n_k)\in \underbrace{U_k \times \cdots \times U_k}_{n}$ and $n_{k}\in \mathbb{N}$ such that $T^{n_k}(y^i_k)=y_i$ for $i=1,\cdots, n.$ Hence $(y^i_k)_k\to 0$ for every $i\in \{1,\cdots,n).$ This proves $(ii)$.\par
	\bigskip
	Conversely, suppose that  $T$ satisfies the given conditions. Let $n\in \mathbb{N}$, $(U_1\times \cdots \times U_n)$ be opene in $X \times \cdots \times X$ and $(y_1,\cdots,y_n)\in X \times \cdots \times X.$ Choose $x_i\in U_i\cap D$ for every $i\in \{1,\cdots,n\}.$ Then from condition $(i)$, there exists $N\in \mathbb{N}$ such that $T^n(x_i)=0$ for every $n\geq N$ and $i\in \{1,\cdots,n\}.$\par
	On the other hand, from condition $(ii)$, there exists a sequence $(n_k)$ and $y^i_{k}\in T^{-n_k}(y_i)$ for every $k\in \mathbb{N}$ such that $(y^i_{k})_k\to 0$ for every $i\in \{1,\cdots,n\}$. Consequently, for large enough $k,$ $x_i+y^i_k\in U_i$ and $T^{n_k}(x+y^i_k)=y_i$ for every $i\in \{1,\cdots,n\}$. Hence $(
	\underbrace{X \times \cdots \times X}_{n},
	\underbrace{T \times \cdots \times T}_{n})$ is strongly transitive and so $X$ is strongly product transitive.
	
\end{proof}

\begin{Remark}
	It is trivial to see that strongly product transitive operators will be strongly transitive.
\end{Remark}

\bigskip

We now consider examples:
    
\begin{Example}
	 Let $H(\mathbb{C})$ be the space of all entire functions on $\mathbb{C}$ endowed with the topology of uniform convergence on compact
		sets. Let $D: H(\mathbb{C}) \to H(\mathbb{C})$, defined by $D(f)=f^{\prime}$ be the derivative operator on $H(\mathbb{C}).$ We show that $(H(\mathbb{C}), D)$ is strongly transitive.
		
		Let $P$ be the set of complex polynomials. Clearly $P$ is dense in $H(\mathbb{C})$ and for every $f\in P$ there exists $n\in \mathbb{N}$ such that $D^n(f)=0$. \par
		On the other hand we define a map $S:H(\mathbb{C})\to H(\mathbb{C})$ by
		$$S(f)(z)=\displaystyle{\int_0^z f(\xi)\ d\xi}$$ where
		$f\in H(\mathbb{C})$ and $z\in \mathbb{C}.$ Clearly $DS(f)=f$ for every $f\in H(\mathbb{C})$.\par
		Let $K$ be a compact subset of $\mathbb{C}$ and $f\in H(\mathbb{C}).$ There exists some $r>0$ such that the closed bounded ball  $B=B[0,r]$ of radius $r$ centered at $0$ contains $K$. Then for every $z\in B$ and $t\in [0,1],$ $zt\in B$. Note that there exists $M>0$ such that $|f(z)|\leq M$ for every $z\in B$. Now for every $z\in B,$
		$$S(f)(z)=\displaystyle{\int_0^z f(\xi)\ d\xi}=\displaystyle{\int_0^1 f(zt) z\ dt}$$ by putting $\xi=zt$. Consequently,  $|S(f)(z)|\leq \displaystyle{\int_0^1 |f(zt)| |z|\ dt}\leq M|z|$ for every $z\in B$. Similarly,
		$$S^2(f)(z)=\displaystyle{\int_0^z Sf(\xi)\ d\xi}=\displaystyle{\int_0^1 Sf(zt) z\ dt}$$ by putting $\xi=zt$. Then $|S^2(f)(z)|\leq \displaystyle{\int_0^1 |Sf(zt)| |z|\ dt}\leq M|z|^2\displaystyle{\int_0^1 |t|\ dt}\leq \frac{M|z|^2}{2!} $ for every $z\in B$. Proceeding this way one can prove that $|S^n(f)(z)|\leq \frac{M|z|^n}{n!} \leq \frac{M r^n}{n!}$ for every $n\in \mathbb{N}$ and $z\in B.$ 
		
		Thus $|S^n(f)(z)|\leq \frac{Mr^n}{n!}$ for every $n\in \mathbb{N}$ and $z\in K.$ Hence the sequence $(S^n(f))$ converges to $0$ uniformly on $K$ and so $S^n(f)\to 0$ in $H(\mathbb{C}).$  Therefore from Theorem \ref{ST}, we can conclude that $D$ is strongly transitive.

Further, note that $(S^n(f))$ converges to $0$ uniformly on $K$ and so for every finite $\{f_1, \ldots, f_n \} \subset P$ and so $S_n(f_i) \to 0$ on $K$ for all $i \in \{1, \ldots, n\}$. Also $D^n(f_i)=0$ for all $i \in \{1, \ldots, n\}$. Thus by Theorem \ref{SPT}, $D$ is strongly product transitive.

 \end{Example}
		
         \begin{Example}
         	Let $B : \ell^p(\mathbb{N}) \to \ell^p(\mathbb{N})$ be the backward
		shift operator, defined by $B(x_0, x_1, \ldots) = (x_1, x_2, \ldots ).$ Let  $D$ be the set of all sequences which are eventually zero, then for every $x\in D$ there exists $n\in \mathbb{N}$ such that $B^n(x)=0$. Also if $|\lambda|>1$ then $(\lambda B)^n(x) = 0$ for all $x \in D$.
		
		Define $S: \ell^p(\mathbb{N}) \to \ell^p(\mathbb{N})$ as $S(x_0, x_1, \ldots) = (0, x_0, x_1, x_2, \ldots ).$ Then $BS(x)=x$ 	 for every $x\in \ell^p$.
		
		Again for $|\lambda|>1$ and every $x\in \ell^p$, $\frac{1}{\lambda^n} S^n(x)\to 0$ as $n\to \infty$. Thus from Theorem \ref{ST} we can conclude that $(\ell^p, \lambda B)$ is strongly transitive for any scalar $\lambda$ with $|\lambda| > 1.$
		
		And from Theorem \ref{SPT}, it follows that $(\ell^p, \lambda B)$ is strongly product transitive for any scalar $\lambda$ with $|\lambda| > 1.$
		
         \end{Example}
         \begin{Example}
         	
         	Recall Example \ref{wei-seq}.
         	
         	 Here $T=B_w$ is not strongly transitive as $T$ is not surjective.
         	 
         	  Note that here the backward orbit of $0$ i.e., $\displaystyle{\bigcup_{n}T^{-n}(0)}$ is dense in $\ell^p(\mathbb{N})$ as the set of all sequence which are eventually zero belong to the set $\displaystyle{\bigcup_{n}T^{-n}(0)}$.
		 \end{Example}

We end with a question here that we leave open:

 \textbf{Question:} Are there strongly transitive operators that are not strongly product transitive in linear dynamics?

	

\bigskip
 


	\begin{thebibliography}{99}
     \bibitem{akin} Ethan Akin, Recurrence in Topological Dynamical Systems: Furstenberg Families and Ellis Actions, Plenum
 Press, New York, 1997.

 \bibitem{AN} Ethan Akin, Joseph Auslander, Anima Nagar, Variations on the concept of topological transitivity, Stud. Math. 235 (2016) 225--249.

		
\bibitem{linear} Frederic Bayart and Etienne Matheron, Dynamics of linear operators, Cambridge Tracts in Mathematics, 179. Cambridge University Press, Cambridge (2009).
		
\bibitem{b1} Bernard Beauzamy, Un opérateur, sur l'espace de Hilbert, dont tous les polynômes sont hypercycliques, C. R. Acad. Sci. Paris Sér. I Math. 303 (1986), 923-925. 


\bibitem{b2} Bernard Beauzamy, An operator on a separable Hilbert space with many hypercyclic vectors, Studia Math. 87(1987), 71-78.  

\bibitem{b3} Bernard Beauzamy, Opérateurs de rayon spectral strictement supérieur à 1, C. R. Acad. Sci. Paris Sér. I. Math. 304 (1987), 263-266.  		
		 
\bibitem{Bes} Juan P. Bès, Three problem’s on hypercyclic operators. Ph.D. thesis, Bowling Green State University, Bowling Green, Ohio, 1998.
		 
\bibitem{Bes2} Juan P. Bès and Alfredo Peris, Hereditarily hypercyclic operators, Journal of Functional Analysis, (1999), 167, 94–112.
		
		
\bibitem{bes ST} Juan P. Bès, Quentin Menet, Alfredo Peris and Yunied Puig, Strong transitivity properties for operators, Journal of Differential Euations, (2019), 266, 1313–1337.
		 
\bibitem{birkhoff} George D. Birkhoff, Démonstration d'un théorème élémentaire sur les fonctions entières, C. R. Acad. Sci. Paris 189 (1929), 473-475. 

\bibitem{CS} George Costakis and Martín Sambarino, Topologically mixing hypercyclic operators,
  Proc. Amer. Math. Soc. 132 (2004), no. 2, 385-389.
		
		
\bibitem{F1} Hillel Furstenberg, Disjointness in ergodic theory, minimal sets, and a problem in Diophantine approximation, Math. Syst. Theory 1 (1967) 1–49.

\bibitem{F2} Hillel Furstenberg, Recurrence in Ergodic Theory and Combinatorial Number Theory, M.B. Porter Lectures, Princeton University Press, Princeton, NJ, 1981.

\bibitem{GS} Gilles Godefroy and Joel H. Shapiro, Operators with dense, invariant, cyclic vector manifolds, J. Funct. Anal.
98 (1991), 229–269


\bibitem{composition} Daniel Gomes and K.-G. Grosse-Erdmann, Kitai’s Criterion for composition operators, Journal of Mathematical Analysis and Applications, 547 (2025),  129347.

\bibitem{GH} Walter H. Gottschalk and Gustav A. Hedlund, Topological dynamics, Amer. Math. Soc. Colloq. Publ.,  Vol. 36, Amer. Math. Soc., Providence, RI, 1955.

\bibitem{mixing example} Sophie Grivaux, Hypercyclic operators, mixing operators, and the bounded steps problem, J. Operator Theory 54 (2005),  147-168.

\bibitem{racsum} K-G. Grosse-Erdmann and Alfredo Peris, Weakly mixing operators on topological vector spaces, RACSAM 104(2) (2010), 413–426.


\bibitem{linear chaos} K.-G. Grosse-Erdmann and Alfredo Peris, Linear Chaos, Universitext, 978-1-4471-2170-1, Springer London (2011).

 

\bibitem{prox1} Wen Huang, Song Shao and Xiangdong Ye, Mixing and proximal cells along sequences, Nonlinearity,  17 (4) (2004), 1245–1260.

        

		
\bibitem{kitai} Carol Kitai, Invariant closed sets for linear operators. Ph.D. thesis, University of Toronto, Toronto, 1982.

\bibitem{gm} Gerald R. MacLane,  Sequences of derivatives and normal families,  
J. Analyse Math. 2 (1952), 72-87.
		
\bibitem{M} T. K. Subrahmonian Moothathu, Orbital and spectral aspects of hypercyclic operators and semigroups, Indag. Math.  30 (2019),  1006-1022.

\bibitem{Anima} Anima Nagar, Revisiting variations in topological transitivity,  
Eur. J. Math. 8 (2022),  369-387.

\bibitem{sr} Stefan Rolewicz,  On orbits of elements, 
Studia Math. 32 (1969), 17-22.    

\bibitem{weight} Hector N. Salas, Hypercyclic weighted shifts, Trans. Amer. Math. Soc., 347 (3), (1995) 993–1004.

\bibitem{prox2} Song Shao, Proximity and distality via Furstenberg families, Topology and its Applications, 153 (2006), 2055-2072.
        
        
\bibitem{shapiro}  Joel H. Shapiro, Notes on the dynamics of linear operators, unpublished lecture notes, available at www.math.msu.edu/~shapiro, 2001.
		
    		
		
		
		
		
		
		
		
		
		
		
	\end{thebibliography}
\end{document}